\documentclass[11pt,a4paper,leqno]{article}
\usepackage{fullpage}
\usepackage{amsmath,amsthm,amssymb}
\usepackage{graphicx}
\usepackage{tikz}
\tikzstyle{every picture} = [>=latex]

\newtheorem{theorem}{Theorem}
\newtheorem{lemma}[theorem]{Lemma}

\def\ca#1{{\mathcal{#1}}}

\title{A Short Proof of Euler--Poincar\'e Formula}

\begin{document}
\author{\bf Petr Hlin\v en\'y
	\\[1ex]Faculty of Informatics, Masaryk University, Brno, Czech
		Republic}

\maketitle

\begin{abstract}
``$V-E+F=2$'', the famous Euler's polyhedral formula,
has a natural generalization to convex polytopes in every finite dimension,
also known as the Euler--Poincar\'e Formula.
We provide another short inductive proof of the general formula. 
Our proof is self-contained and it does not use shellability of polytopes.

{\bf Keywords:} {Euler--Poincar\'e formula, polytopes, discharging}
%
\end{abstract}

\section{Introduction}

In this paper we follow the standard terminology of polytopes theory,
such as Ziegler~\cite{Ziegler}.
We consider {\em convex polytopes}, defined as a convex hull of finitely
many points, in the $d$-dimensional Euclidean space for an arbitrary
$d\in\mathbb N$, $d\geq1$.
We shortly say a polytope to mean a convex polytope.
A landmark discovery in the history of combinatorial investigation 
of polytopes was famous Euler's formula, 
stating that for any $3$-dimensional polytope with $v$ vertices,
$e$ edges and $f$ faces, $v-e+f=2$ holds.
This finding was later generalized, in every dimension~$d$,
to what is nowadays known as (generalized) Euler's relation or
Euler--Poincar\'e formula, as follows.

For instance, in dimension $d=1$ we have $v=2$, which can be rewritten as
$v-1=1$, and in dimension $d=2$ we have got $v-e=0$ or $v-e+1=1$.
Similarly, the $d=3$ case can be rewritten as $v-e+f-1=1$.
Note that the `$1$' left of `$=$' stands in these expressions for the polytope itself.
In general, the following holds:

\begin{theorem}[``Euler--Poincar\'e formula''; 
	Schl\"afli \cite{Schlafli} 1852]\label{thm:EulerPoin}
Let $P$ be a convex polytope in~$\mathbb R^d$, 
and denote by $f^c$, $c\in\{0,1,\dots,d\}$,
the numbers of faces of $P$ of dimension~$c$. Then
\begin{equation}\label{eq:EulerPoin}
	 f^0-f^1+f^2-\dots + (-1)^d f^d = 1 
.\end{equation}
\end{theorem}

We refer to classical textbooks of Gr\"unbaum~\cite{Grunbaum}
and Ziegler~\cite{Ziegler} for a closer discussion of the interesting
history of this formula and of the difficulties associated with its proof.
Here we just briefly remark that all the historical attempts to prove the
formula in a combinatorial way, starting from Schl\"afli, 
implicitly assumed validity of a
special property called {\em shellability} of a polytope,
However, the shellability of any polytope was formally established
only in 1971 by Bruggesser and Mani~\cite{Bruggesser-Mani}.

We provide a new self-contained inductive proof of \eqref{eq:EulerPoin} 
which does not assume shellability.
Our proof has been in parts inspired by a proof of $3$-dimensional
Euler's formula via angles 
\cite[``Proof~8: Sum of Angles'']{junkyardEuler},
and by Welzl's probabilistic proof \cite{DBLP:conf/birthday/Welzl94}
of Gram's equation -- although,
the resulting exposition of the proof does not resemble either of those.
In fact, we give two different expositions of the new proof,
one in Section~\ref{sec:proofI} and the other in Section~\ref{sec:proofII}.

\section{New Self-contained Proof}
\label{sec:proofI}

Our new proof of Theorem~\ref{thm:EulerPoin} proceeds by induction on
the dimension $d\geq1$.
Note that validity of \eqref{eq:EulerPoin} is trivial for $d=1,2$,
and hence it is enough to show the following:

\begin{figure}
\hbox to \hsize{%
\hfill\includegraphics[width=0.35\hsize]{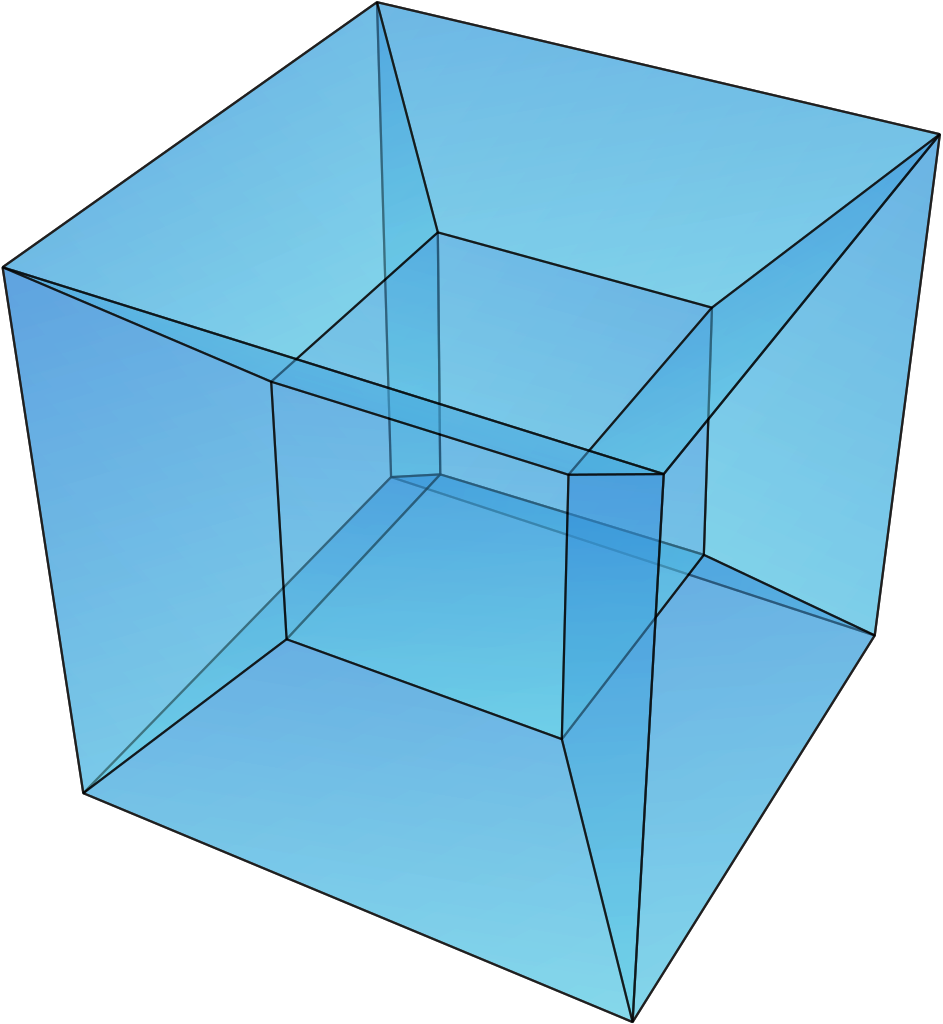}
\hfill}%
\caption{An illustration: the Schlegel diagram of the $4$-cube,
	consisting of the outer $3$-cube subdivided into seven combinatorial
	$3$-cubes. (Credit: Goffrie\,/\,Wikimedia Commons)}
\label{fig:wireframe}
\end{figure}

\begin{lemma}\label{lem:EulerPoinI}
Let $k\geq2$ and $P$ be a polytope of dimension $k+1$.
Assume that \eqref{eq:EulerPoin} holds for any polytope of
dimension $d\in\{k-1,k\}$.
Then \eqref{eq:EulerPoin} holds for $P$ (with $d=k+1$).
\end{lemma}

\begin{proof}
Recall that $f^c$, $c\in\{0,1,\dots,k+1\}$, denote
the numbers of faces of $P$ of dimension~$c$.
Our goal is to prove
\begin{equation*}
         f^0-f^1+f^2-\dots + (-1)^{k-1}f^{k-1} + (-1)^k f^k + (-1)^{k+1}f^{k+1} = 1
,\end{equation*}
or equivalently, since $f^{k+1}=1$,
\begin{equation}\label{eq:needed}
         f^0-f^1+f^2-\dots + (-1)^{k-1}f^{k-1}  = 1 + (-1)^k (1-f^k)
.\end{equation}

Let $\ca R$ denote a $k$-dimensional polyhedral complex which is the Schlegel
diagram of $P$ with respect to its facet $T$ 
(i.e., $\ca R$ results from $P$ by projecting it to $T$ from a point
outside of $P$ but sufficiently close to some interior point of~$T$;
see Figure~\ref{fig:wireframe}).
Let the ($k$-dimensional) facets of $\ca R$, 
which are the projections of the facets of $P$ other than~$T$, 
be denoted by $R_1,\dots,R_a$ where~$a:=f^k-1$, in any order, and let~$R_0:=T$.
Clearly, for each $c\in\{0,1,\dots,k-1\}$, the number of faces of dimension $c$ 
in the complex $\ca R$ is exactly~$f^c$.

The high level idea of our proof is to double-count certain objects
(signed flags) distributed to all the faces of~$\ca R$ of dimensions from
$0$ to~$k-1$.

We choose any straight line $q$ which is in {\em general position} 
with respect to~$\ca R$.
In particular, $q$ is not parallel to any face of~$\ca R$.
For every face $F$ of $\ca R$ of dimension $0\leq c\leq k-1$,
we select a point $x_F$ in the relative interior of~$F$, and we add from
$x_F$ two ``flags'' in the opposite directions of line~$q$.
Formally, these {\em flags} are the two line segments $s_1,s_2$ 
with the common end $x_F$, $s_1\cap s_2=\{x_F\}$,
which are parallel to $q$ and of length $\varepsilon>0$.
We call $x_F$ the {\em base point} and $F$ the {\em base face} 
of the flags $s_1,s_2$.
The length $\varepsilon$ is chosen to be smaller than the minimum
distance between two disjoint faces of~$\ca R$.
Each flag whose base face is of dimension $c$ gets the value
of~$\frac12(-1)^c$.

It is now clear from the definition that the total value of all flags 
distributed in~$\ca R$ equals the left-hand side of \eqref{eq:needed}.
Another easy observation is that, since $q$ has been chosen in general
direction, every flag $s$ is either contained in a unique facet of $\ca R$,
or $s$ except its base point is disjoint from the point set of $\ca R$
(in our case, actually, $set(\ca R)=R_0$).
Furthermore, no two flags sharing the same base point are contained in the
same facet of $\ca R$, by convexity.

Our next task is to sum the total value of the flags in each one facet
$R_i$ of~$\ca R$ (see~\eqref{eq:insflag}), 
and also of the flags that are outside of~$R_0$ (see~\eqref{eq:outflag}).
Consider a facet $R_i$ of~$\ca R$, $i\in\{1,\dots,a\}$
and a face $F$ of $R_i$ of dimension~$c$.
The following observation follows readily from the assumption of a
general position of $q$ and from convexity:
\begin{itemize}
\refstepcounter{equation}\label{eq:project}
\item[(\theequation)]{\it
If $c=k-1$, then always one of the two flags with the base face $F$ belongs
to $R_i$.
If $0\leq c\leq k-2$, then one of the two flags with the base $F$ belongs
to $R_i$ if, and only if, the image of $F$ is not a face in the orthogonal
projection of $R_i$ along~$q$.
}\end{itemize}

We let $S_i$, $i\in\{0,1,\dots,a\}$, denote the 
$(k-1)$-dimensional image of $R_i$ in the orthogonal projection along~$q$.
Furthermore, let $f^c_i$, where $c\in\{0,1,\dots,k\}$ and $i\in\{0,1,\dots,a\}$,
denote the number of faces of $R_i$ of dimension~$c$,
and $g^c_i$, $c\in\{0,1,\dots,k-1\}$,
denote the number of faces of $S_i$ of dimension~$c$.
Now by \eqref{eq:project}, the number of flags with a base face of dimension
$c$ which are contained in $R_i$, is the following:
$f_i^c$ if $c=k-1$, and $f_i^c-g^c_i$ if $0\leq c\leq k-2$.
Hence the overall value of all the flags contained in $R_i$ is
\begin{eqnarray}
\frac12(-1)^{k-1}f_i^{k-1} &+& \frac12\sum_{c=0}^{k-2}(-1)^c(f_i^c-g^c_i)
 	~=~ \frac12\sum_{c=0}^{k-1} (-1)^c f_i^c - \frac12\sum_{c=0}^{k-2}(-1)^c g^c_i
\nonumber \\	\label{eq:insflag}
  &=& \frac12\left(1-(-1)^kf_i^k\right) - \frac12\left(1-(-1)^{k-1}g_i^{k-1}\right)
	= (-1)^{k-1}
,\end{eqnarray}
where the latter \eqref{eq:insflag} holds true by applying 
\eqref{eq:EulerPoin} to the polytopes $R_i$ and
$S_i$ of dimensions $d=k$ and $k-1$, respectively,
and by obvious~$f_i^k=g_i^{k-1}=1$.

The overall value of all the remaining flags, which are not in $R_0$,
is similarly equal to
\begin{equation}\label{eq:outflag}
\frac12\sum_{c=0}^{k-1} (-1)^c f_0^c \>+\> \frac12\sum_{c=0}^{k-2}(-1)^c g^c_0
  = \frac12\left(1-(-1)^kf_0^k\right) + \frac12\left(1-(-1)^{k-1}g_0^{k-1}\right)
   = 1
.\end{equation}

Since we have been counting twice the same set of flags,
we get that (the left-hand side of) \eqref{eq:needed}
must equal the sum of \eqref{eq:insflag} over $i=1,\dots,a$, and of
\eqref{eq:outflag}, leading to
\begin{equation*}
f^0-f^1+f^2-\dots + (-1)^{k-1}f^{k-1} = (-1)^{k-1}a+1 
  = (-1)^{k}(1-f^k)+1
,\end{equation*}
and thus finishing the proof of \eqref{eq:needed} for $P$ in dimension~$k+1$.
\end{proof}

\section{Another Point of View}
\label{sec:proofII}

Interestingly, there is another and quite different exposition of the proof
ideas from Section~\ref{sec:proofI}, which deals directly with the polytope
$P$ instead of its Schlegel diagram, and which uses only very elementary
and apparent combinatorial and geometric arguments.
This exposition has also been published (in full)
in the proceedings~\cite{10.1007/978-3-030-83823-2_15}.

\begin{proof}[\bfseries Alternative proof of Lemma~\ref{lem:EulerPoinI}]
Recall that $P$ is a polytope in dimension $k+1$, and
$f^c$, $c\in\{0,1,\dots,k+1\}$, denote the numbers of faces of $P$ of dimension~$c$.
Again, we assume validity of \eqref{eq:EulerPoin} in dimensions $d\in\{k-1,k\}$,
and our goal is to prove
\begin{equation}\label{eq:neededII}
	 f^0-f^1+f^2-\dots + (-1)^k f^k + (-1)^{k+1} f^{k+1} = 1 
.\end{equation}

We choose arbitrary two facets $T_1,T_2$ of $P$ (distinct, but not
necessarily disjoint) and two points $t_1\in T_1$ and $t_2\in T_2$ in their
relative interior, such that the straight line $q$ determined by $t_1,t_2$
is in a general position with respect to~$P$.
In particular, we demand that {\em no} nontrivial line segment lying
in a face of $P$ of dimension $\leq k-1$ is coplanar with~$q$.
We also denote by $T_3,\dots,T_{f^k}$ the remaining facets of $P$, in any order.
As in Section~\ref{sec:proofI},
for every face $F$ of $P$ of dimension $0\leq c\leq k-1$,
we select a point $x_F$ in the relative interior of~$F$
(note that $x_v=v$ if $v$ is a~vertex~of~$P$).

\smallskip
In the proof we use a {\em discharging} argument, an advanced variant of
the double-counting method in combinatorics.
To every face $F$ of $P$ of dimension $0\leq c\leq k+1$,
we assign {\em charge} of value~$(-1)^c$ (note that charge applies also to
whole~$F=P$).
Hence the {\em total change} initially assigned to all faces of $P$ equals the
left-hand side of \eqref{eq:neededII}.

\begin{figure}[t]
$$
\begin{tikzpicture}[xscale=1.3]
\tikzstyle{every node}=[draw, shape=circle, inner sep=1.2pt, fill=white]
\tikzstyle{every path}=[draw=black]
\draw (-0.5,1.5) node[fill=black,label=left:{$t_i$}] (t) {} ;
\tikzstyle{every node}=[draw=blue, shape=circle, inner sep=1.2pt, fill=white]
\fill[fill=lightgray]
        (0.5,0) node[label=below:{$a$}] (a) {} --
	  (2.5,0) node[label=below:{$b$}] (b) {} -- 
	(4,1) node[label=right:{$c$}] (c) {} --
	  (3,3) node[label=above:{$d$}] (d) {} -- 
	(1,3) node[label=above:{$e$}] (e) {} -- (a) ;
\tikzstyle{every node}=[draw=red, shape=circle, inner sep=1.2pt, fill=white]
\draw (1.5,0) node[label=below:{$x_A$}] (A) {} ;
\draw (3,0.3) node[label=right:{~~$x_B$}] (B) {} ;
\draw (3.5,2) node[label=right:{~$x_C$}] (C) {} ;
\draw (2,3) node[label=above:{$x_D$}] (D) {} ;
\draw (0.56,0.4) node[label=left:{$x_E$}] (E) {} ;
\tikzstyle{every path}=[dashed,draw=brown]
\draw (a) -- (t) -- (b) ;
\draw (c) -- (t) -- (d) ;
\draw (t) -- (e) -- (1.3,3.3) ;
\draw (A) -- (t) -- (B) ;
\draw (C) -- (t) -- (D) ;
\draw (t) -- (E) -- (0.8,0.15) ;
\tikzstyle{every path}=[draw]
\draw[->, draw=blue] (b) -- (2.3,0.7) node[draw=none,fill=none,
		 label=left:{$+\frac12\!\!\!$}] {}  ;
\draw[->, draw=blue] (c) -- (3.3,1.2) ;
\draw[->, draw=blue] (d) -- (2.7,2.4) ;
\draw[->, draw=red] (D) -- (2,2.3) node[draw=none,fill=none,
		 label=left:{$-\frac12\!\!\!$}] {}  ;
\draw[->, draw=red] (A) -- (1.6,0.6) ;
\draw[->, draw=red] (E) -- (1.1,0.8) ;
\draw[->, draw=red] (C) -- (3.1,1.8) ;
\draw[->, draw=red] (B) -- (2.75,0.8) ;
\tikzstyle{every node}=[draw=none]
\draw (2,1.5) node {$T_i$} ;
\end{tikzpicture}
\qquad\qquad
\begin{tikzpicture}[xscale=1.3]
\tikzstyle{every node}=[draw, shape=circle, inner sep=1.2pt, fill=white]
\tikzstyle{every path}=[draw=black]
\tikzstyle{every node}=[draw=blue, shape=circle, inner sep=1.2pt, fill=white]
\fill[fill=lightgray]
        (0.5,0) node[label=below:{$a$}] (a) {} --
	  (2.5,0) node[label=below:{$b$}] (b) {} -- 
	(4,1) node[label=right:{$c$}] (c) {} --
	  (3,3) node[label=above:{$d$}] (d) {} -- 
	(1,3) node[label=above:{$e$}] (e) {} -- (a) ;
\draw (2,1.5) node[fill=black,label=left:{$t_2$}] (t) {} ;
\tikzstyle{every node}=[draw=red, shape=circle, inner sep=1.2pt, fill=white]
\draw (1.5,0) node[label=below:{$x_A$}] (A) {} ;
\draw (3,0.3) node[label=right:{~~$x_B$}] (B) {} ;
\draw (3.5,2) node[label=right:{~$x_C$}] (C) {} ;
\draw (2,3) node[label=above:{$x_D$}] (D) {} ;
\draw (0.75,1.5) node[label=left:{$x_E$}] (E) {} ;
\tikzstyle{every path}=[dashed,draw=brown]
\draw (a) -- (t) -- (b) ;
\draw (c) -- (t) -- (d) ;
\draw (t) -- (e) ;
\draw (A) -- (t) -- (B) ;
\draw (C) -- (t) -- (D) ;
\draw (t) -- (E) ;
\tikzstyle{every path}=[draw]
\draw[->, draw=blue] (b) -- (2.2,0.7) node[draw=none,fill=none,
		 label=right:{$\!\!\!\!+\frac12$}] {}  ;
\draw[->, draw=blue] (a) -- (1,0.4) ;
\draw[->, draw=blue] (c) -- (3.4,1.1) ;
\draw[->, draw=blue] (d) -- (2.7,2.4) ;
\draw[->, draw=blue] (e) -- (1.5,2.4) ;
\draw[->, draw=red] (D) -- (1.95,2.3) node[draw=none,fill=none,
		 label=right:{$\!\!\!\!-\frac12$}] {}  ;
\draw[->, draw=red] (A) -- (1.6,0.6) ;
\draw[->, draw=red] (E) -- (1.3,1.4) ;
\draw[->, draw=red] (C) -- (3.1,1.8) ;
\draw[->, draw=red] (B) -- (2.75,0.8) ;
\tikzstyle{every node}=[draw=none]
\draw (3,1.5) node {$T_2$} ;
\end{tikzpicture}
$$
\caption{Proof of Lemma~\ref{lem:EulerPoinI}:
    a facet in a $3$-dimensional polytope~$P$ ($k=2$).
    Each vertex and facet of $P$ initially get charge of~$1$ and each edge~$-1$.
    Consider, e.g., a facet $T_i$ of $P$ which is a pentagon with vertices
    $a,b,c,d,e$ and sides (edges) $A,B,C,D,E$ in order.
    Let $t_i$ be the point in which the plane of $T_i$ intersects the line~$q$
    (see in the proof).
    On the left of the picture ($t_i\not\in T_i$, for $i\geq3$), 
    we have that the vertices $b,c,d$ send
    charge of $\frac12$ to $T_i$ by the rule \eqref{eq:discharge},
    while $a,e$ are not sending to~$T_i$.
    On the right ($t_i\in T_i$, $i=1,2$), all the vertices $a,b,c,d,e$ send charge of
    $\frac12$ to $T_i$.
    In both cases, every side $A,B,C,D,E$ sends charge of $-\frac12$ to $T_i$.
    Consequently, on the left $T_i$ ends up with charge $0$
    (compare to \eqref{eq:insflagIII}), while on the
    right with charge~$1$ (cf.\eqref{eq:twospecIII}).
    }
\label{fig:discharge}
\end{figure}

Now we {\em discharge} all the assigned charge from faces
of dimension $c\leq k-1$ to the facets of $P$.
The discharging rule is only one and quite simple.
Consider a facet $T_i$ of $P$, $1\leq i\leq f^k$.
Let $t_i\in q$ denote the unique point which is the intersection of the line
$q$ with the support hyperplane of~$T_i$.
This is a sound definition of $t_i$ according to a general position of~$q$,
and it is consistent with the choice of $t_1,t_2$ above.
For any proper face $F$ of $T_i$ (so $F$ is a face of $P$ as
well and is of dimension $0\leq c\leq k-1$), 
the discharging rule reads (see in Figure~\ref{fig:discharge}):
\begin{equation}\label{eq:discharge}
\vtop{\advance\hsize-6ex\it\noindent
The face $F$ sends half of its initial charge, i.e.~$\frac12(-1)^c$,
to the facet $T_i$ if, and only if,
the straight line passing through $x_F$ and~$t_i$ intersects the relative
interior of~$T_i$. 
}\end{equation}
Note that we will be finished if we prove that, after applying the discharging rule,
(i)~every face of $P$ of dimension $\leq k-1$ ends up with charge $0$,
and (ii) the total charge of the facets of $P$ and of $P$ itself
sums up to the right-hand side of \eqref{eq:neededII}.

\begin{figure}[t]
$$
\begin{tikzpicture}[scale=0.8]
\tikzstyle{every path}=[dashed,draw=brown]
\draw (3.5,3) node (x) {} -- (-2,1) node (ti1) {} ;
\draw (x) -- (7,1) node (ti2) {} ;
\tikzstyle{every path}=[draw=black]
\tikzstyle{every node}=[draw=none]
\fill[fill=lightgray]
        (2,0) node (a) {} -- (4.5,0) -- (5.45,1.89) --
	(3.5,3) node (x) {} -- (2,2.46) -- (1,1.5) -- (2,0) ;
\draw (2.5,3) node {$A_1$} ;
\draw (4.3,3) node {$A_2$} ;
\draw (3.5,1.5) node {$N$} ;
\draw (-3,1) node[label=below:{$q$}] {} -- (8,1) ;
\tikzstyle{every node}=[draw, shape=circle, inner sep=1.2pt, fill=white]
\draw (x) node[label=above:{$x_F$}] {} ;
\draw (1.33,1) node[label=below:{$t_1$}] {} ;
\draw (5,1) node[label=below:{~$t_2$}] {} ;
\draw (ti1) node[fill=black, label=below:{$t_{i_1}$}] {} ;
\draw (ti2) node[fill=black, label=below:{$t_{i_2}$}] {} ;
\end{tikzpicture}\qquad\qquad\qquad\qquad
$$ $$\qquad\qquad\qquad\qquad
\begin{tikzpicture}[scale=0.8]
\tikzstyle{every path}=[dashed,draw=brown]
\draw (2,2.46) node (x) {} -- (-2,1) node (ti1) {} ;
\draw (x) -- (0.5,1) node (ti2) {} ;
\tikzstyle{every path}=[draw=black]
\tikzstyle{every node}=[draw=none]
\fill[fill=lightgray]
        (2,0) node (a) {} -- (4.5,0) -- (5.45,1.89) --
	(3.5,3) -- (2,2.46) node (x) {} -- (1,1.5) -- (2,0) ;
\draw (2.9,3.2) node {$A_1$} ;
\draw (0.9,2) node {$A_2$} ;
\draw (3.5,1.5) node {$N$} ;
\draw (-4,1) node[label=below:{$q$}] {} -- (6,1) ;
\tikzstyle{every node}=[draw, shape=circle, inner sep=1.2pt, fill=white]
\draw (x) node[label=above:{$x_F$}] {} ;
\draw (1.33,1) node[label=below:{$t_1$}] {} ;
\draw (5,1) node[label=below:{~$t_2$}] {} ;
\draw (ti1) node[fill=black, label=below:{$t_{i_1}$}] {} ;
\draw (ti2) node[fill=black, label=below:{$t_{i_2}$}] {} ;
\end{tikzpicture}
$$
\caption{Proof of Lemma~\ref{lem:EulerPoinI}:
    Two cases of a polygon $N$ which is the intersection of the polytope~$P$ with
    the plane spanning $q$ and a point $x_F$ of a face~$F$. 
    The two sides $A_1,A_2$ of $N$ incident to~$x_F$ determine the two
    unique facets $T_{i_1},T_{i_2}$ of $P$ that $F$ sends charge to.
    }
\label{fig:planecut}
\end{figure}

\smallskip
For the task (i), consider any face $F$ of $P$ of dimension $c\leq k-1$
and the point $x_F$ chosen in $F$ above.
Let $L$ denote the plane determined by the line $q=\overline{t_1t_2}$
and the point $x_F\not\in q$.
Then $N:=P\cap L$ is a convex polygon.
See Figure~\ref{fig:planecut}.
We claim that $x_F$ must be a vertex of~$N$:
indeed, if $x_F$ belonged to a relative interior of a side $A_0$ of~$N$,
then $A_0\subseteq F$ and $A_0$ would be coplanar with~$q$, contradicting
our assumption of a general position of~$q$.
Consequently, $x_F$ is incident to two sides $A_1,A_2$ of~$N$, and
there exist facets $T_{i_1},T_{i_2}$ of $P$,
$1\leq i_1\not=i_2\leq f^k$, such that $A_j=T_{i_j}\cap L$ for $j=1,2$.
Observe that the support line of $A_j$ intersects $q$ precisely in 
$t_{i_j}$ (which has been defined as the intersection of the 
support hyperplane of $T_{i_j}$ with~$q$).

Moreover, since $A_j$ is coplanar with $q$, by our assumption of a general
position of~$q$ it cannot happen that $A_j$ is contained in a face of
dimension $\leq k-1$.
Consequently, $A_j$ (except its ends) belongs to the relative interior 
of $T_{i_j}$, and $T_{i_j}$ is a unique such face for~$A_j$.
Hence, taking this argument for $j=1,2$, we see that $F$ sends away
by \eqref{eq:discharge} exactly 
two halves of its initial charge, ending up with charge~$0$.

\smallskip
For the task (ii), 
let $f^c_i$, where $c\in\{0,1,\dots,k\}$ and $i\in\{1,\dots,f^k\}$,
denote the number of faces of $T_i$ of dimension~$c$.
We first look at the two special facets $T_i$, $i=1,2$
(Figure~\ref{fig:discharge} right).
Since $t_i\in T_i$ in this case, by \eqref{eq:discharge} $T_i$ receives
(half) charge from every of its proper faces.
Finally, $T_i$ starts with charge $(-1)^k$ which we halve in the computation.
Using \eqref{eq:EulerPoin} for $T_i$, which is of dimension $k$ and so $f_i^k=1$,
we thus get that the total charge $T_i$ ends up with, is
\begin{equation}\label{eq:twospecIII}
  \frac 12 \left( f_i^0-f_i^1+\dots + (-1)^{k-1}f_i^{k-1} + (-1)^{k}f_i^{k}
	\right) + \frac 12 (-1)^{k} = \frac 12  + \frac 12 (-1)^{k}
.\end{equation}

Second, consider a facet $T_i$ where $i\geq3$.
Let $H_i$ be the support hyperplane of~$T_i$.
Then $\{t_i\}=H_i\cap q$ and $t_i\not\in T_i$.
We restrict ourselves to the affine space formed by $H_i$, and denote 
by $S_i$ a projection of $T_i$ from the point $t_i$ onto a suitable
hyperplane within~$H_i$.
Since $t_i$ is in a general position with respect to $T_i$ (which is implied
by a general position of $q$), the following holds:
every proper face of $S_i$ is the image of an equivalent face of $T_i$
(of the same dimension).
Furthermore, by convexity, a face $F$ of $T_i$ has no image among the faces
of $S_i$ if, and only if, the line through $x_F$ and~$t_i$ intersects the
relative interior of~$T_i$.
See also Figure~\ref{fig:discharge} left.

Consequently, as directed by \eqref{eq:discharge},
$T_i$ receives charge precisely from those of its faces $F$
which do not have an image among the proper faces of~$S_i$.
Denote by $g^c_i$ the number of faces of $S_i$ of dimension~$c\leq k-1$,
and notice that $f_i^k=g_i^{k-1}=1$.
Hence, precisely, $T_i$ receives $\frac12(-1)^{k-1}$ of charge from
each of its $f_i^{k-1}$ faces of dimension $k-1$, and $\frac12(-1)^{c}$
from $f_i^c-g^c_i$ of its faces of dimension $0\leq c\leq k-2$.
Recall also that $T_i$ itself starts with charge $(-1)^k$.
Summing together, and using \eqref{eq:EulerPoin} 
for $T_i$ (of dimension $k$) and for $S_i$ (of dimension~$k-1$), we get
\begin{eqnarray}
\frac12(-1)^{k-1}f_i^{k-1} &+& \frac12\sum_{c=0}^{k-2}(-1)^c(f_i^c-g^c_i)
	~+~ \frac12(-1)^k+\frac12(-1)^k
\nonumber \\
 	&=& \frac12\sum_{c=0}^{k-1} (-1)^c f_i^c - \frac12\sum_{c=0}^{k-2}(-1)^c g^c_i
	~+~ \frac12(-1)^k-\frac12(-1)^{k-1}
\nonumber \\	\label{eq:insflagIII}
 	&=& \frac12\sum_{c=0}^{k} (-1)^c f_i^c - \frac12\sum_{c=0}^{k-1}(-1)^c g^c_i
	~=~ \frac12-\frac12 = 0
.\end{eqnarray}

Since the total charge is not changed,
we get that the left-hand side of~\eqref{eq:neededII}
must equal the sum of the charge of $P$, of \eqref{eq:twospecIII} over $i=1,2$ and of
\eqref{eq:insflagIII} over remaining facets, leading~to
$$
(-1)^{k+1}+ \frac12+\frac12(-1)^{k} + \frac12+\frac12(-1)^{k} +0 = 1
,$$
and thus finishing the proof of \eqref{eq:neededII} for $P$.
\end{proof}

%

\bibliographystyle{plain}
\bibliography{Euler}

\end{document}